\documentclass[12pt,amscd]{amsart}
\usepackage[all]{xy}
\usepackage{graphicx}
\usepackage{amsmath,amsxtra,amssymb,latexsym, amscd,amsthm}
\usepackage{indentfirst}
\usepackage[mathscr]{eucal}
\usepackage[pagebackref=true]{hyperref}

\newtheorem{thm}{Theorem}[section]
\newtheorem{cor}[thm]{Corollary}
\newtheorem{lem}[thm]{Lemma}

\theoremstyle{definition}
\newtheorem{defn}[thm]{Definition}
\newtheorem{exm}[thm]{Example}

\newtheorem{rem}[thm]{Remark}

\newtheorem{conj}{Conjecture}

\DeclareMathOperator{\NN}{\mathbb {N}}
\DeclareMathOperator{\ZZ}{\mathbb {Z}}

\DeclareMathOperator{\lk}{lk}

\DeclareMathOperator{\supp}{supp}

\DeclareMathOperator{\reg}{reg}

\def\D {\Delta}

\def\a {\mathbf a}
\def\b {\mathbf b}

\def\o {\mathbf {0}}

\def\m {\mathfrak m}
\def\F {\mathfrak F}

\def\h {\widetilde{H}}

\begin{document}

\title[Comparision between regularity of small powers of an edge ideal] {Comparision between regularity of small symbolic powers and ordinary powers of an edge ideal}

\author{Nguyen Cong Minh}
\address{Department of Mathematics, Hanoi National University of Education, 136 Xuan Thuy, Hanoi,
	Vietnam}
\email{minhnc@hnue.edu.vn}
\author{Le Dinh Nam}
\address{School of Applied Mathematics and Informatics, Hanoi University of Science and Technology, 1 Dai Co Viet Road, Hanoi, Vietnam}
\email{nam.ledinh@hust.edu.vn}
\author{Thieu Dinh Phong}
\address{Faculty of Mathematics, Vinh University, 182 Le Duan Street, Vinh city, Vietnam}
\email{thieudinhphong@gmail.com}
\author{Phan Thi Thuy}
\address{Department of Mathematics, Hanoi National University of Education, 136 Xuan Thuy, Hanoi, Vietnam}
\email{thuysp1@gmail.com}  
\author{Thanh Vu}
\email{vuqthanh@gmail.com}


\date{}

\dedicatory{}
\commby{}
\maketitle
\begin{abstract}
    Let $G$ be a simple graph and $I$ its edge ideal. We prove that 
    $$\reg(I^{(s)}) = \reg(I^s)$$
    for $s = 2,3$, where $I^{(s)}$ is the $s$-th symbolic power of $I$. As a consequence, we prove the following bounds
    \begin{align*}
        \reg I^{s} & \le \reg I + 2s - 2, \text{ for } s = 2,3,\\
        \reg I^{(s)} & \le \reg I + 2s - 2, \text{ for } s = 2,3,4.
    \end{align*}
\end{abstract}

\maketitle

\section{Introduction}\label{sect_intro}
The Castelnuovo-Mumford regularity (or regularity for short) is an important invariant of graded algebras. It bounds the maximum degree of the syzygies and the maximum non-vanishing degree of the local cohomology modules. It is a celebrated result that the regularity of $I^s$ is asymptotically a linear function for any homogeneous ideal $I$ in a polynomial ring $S$ over a field (see \cite{CHT,K}). It is a natural question then to ask whether a similar result holds for symbolic powers of $I$. In general, Cutkosky \cite{Cu} gave an example of a homogeneous ideal $I$ such that $\lim_{t\longrightarrow \infty}\dfrac{\reg(I^{(t)})}{t}$ is not rational. So $\reg(I^{(t)})$ is in general far from being asymptotically a linear function. For a monomial ideal Herzog, Hoa, and N. V. Trung \cite{HHT} showed that the regularity of symbolic powers is bounded by a linear function. In recent work, Dung, Hien, Nguyen, and T. N. Trung \cite{DHNT} have constructed a class of squarefree monomial ideals for which $\reg(I^{(t)})$ is not asymptotically a linear function. On the other hand, when $I$ is a Stanley-Reisner ideal of a matroid or a simplicial complex of dimension one then $\reg(I^{(t)})$  is a linear function of $t$ (see \cite{HTr,MTr}). It is not known whether the regularity of symbolic powers of edge ideals of graphs is asymptotically a linear function. More exactly, in this case, the first author raised the following conjecture.

\begin{conj}\label{conj_A}  Let $I(G)$ be the edge ideal of a simple graph $G$. Then for all $s \ge 1$, $$\reg(I(G)^{(s)})=\reg(I(G)^s).$$
\end{conj}

It is noted that the graph $G$ is a bipartite graph if and only if $I(G)^{(s)}=I(G)^s$ for all $s\ge 1$ (\cite[Theorem 5.9]{SVV}). Thus, the above conjecture is trivially true in this case. If $G$ is not bipartite, then it must contain an odd cycle. Gu, Ha, O'Rourke, and Skelton \cite{GHOS} took the first step in verifying this conjecture for odd cycle graphs. Then, Jayanthan and Kumar \cite{JK} proved this conjecture for graphs obtained by the clique sum of odd cycles and bipartite graphs. In recent work, Fakhari \cite{F1, F2, F3} established this conjecture for unicyclic graphs, Cameron-Walker graphs, and chordal graphs. 

In this paper, we prove 
\begin{thm}\label{main_thm}Let $I(G)$ be the edge ideal of a simple graph $G$. Then
$$\reg I(G)^{(s)} = \reg (I(G)^s)$$
for $s = 2, 3$.
\end{thm} 

In other words, we establish Conjecture \ref{conj_A}  for $s = 2, 3$. We would like to note that, in all cases where the regularity of symbolic powers of $I$ was computed, the main technical step was to bound the regularity of certain colon ideals. We do not know of any direct comparison between the regularity of powers and symbolic powers of ideals when the regularity of the corresponding symbolic/ordinary power is unknown. We will now outline the idea of proof of Theorem \ref{main_thm}. 
\begin{enumerate}
    \item We reduce the problem of comparing the regularity of two monomial ideals to the problem of comparing radicals of the colon of these ideals by certain monomials, see Lemma \ref{extremal_red} and Lemma \ref{Key1}.
    \item By Lemma \ref{red0} and induction, we further reduce to studying degree complexes of symbolic powers/ordinary powers of edge ideals of special exponents. We then analyze these degree complexes in detail via the Stanley-Reisner correspondence. 
\end{enumerate}
This procedure for comparing regularity of two monomial ideals is especially useful when the two ideals are closely related; for example, an ideal versus its integral closure, various types of powers of an ideal. Furthermore, our study of degree complexes of symbolic/ordinary powers reveals interesting information on the extremal exponents of powers of edge ideals which will be exploited further in subsequent work to study the regularity of powers of edge ideals themselves.

The main obstructions to proceed the comparison further with higher powers are:
\begin{enumerate}
    \item Explicit description of symbolic powers of higher powers is unknown.
    \item Even in the case where an explicit description of symbolic powers is known, e.g. the case of perfect graphs, the number of critical exponents grows and the radical ideals of colon ideals of powers with respect to these critical exponents are difficult to compute.
\end{enumerate} 
We would like to note further that a weaker form of Conjecture \ref{conj_A}, namely, $\reg I^{(s)} \le \reg I^s$ is easier and might be able to carry a bit further to higher powers.

By combining a recent result of Fakhari \cite[Theorem 3.6]{F4}, we establish a conjecture of Alilooee, Banerjee, Beyarslan, and Ha
 \cite[Conjecture 1]{BBH} for the second and third powers of edge ideals. 

\begin{thm} Let $I(G)$ be the edge ideal of a simple graph $G$. Then   $$\reg(I(G)^{s})\le 2s-2+\reg(I(G)),$$
for $s = 2,3$.
\end{thm}

Note that Banerjee and Nevo \cite{BN} prove this Theorem for $s=2$ by using a topological method. 

For symbolic powers, we extend \cite[Corollary 3.9]{F4} to prove

\begin{thm} Let $I(G)$ be the edge ideal of a simple graph $G$. Then   $$\reg(I(G)^{(s)})\le 2s-2+\reg(I(G)),$$
for $s = 2,3,4$.
\end{thm}

Finally, we obtain explicit values of the regularity of small symbolic powers of $I(G)$ for some new classes of graphs.
  
Now we explain the organization of the paper. In Section 2, we recall some notation and basic facts about the symbolic powers of a squarefree monomial ideal, the degree complexes, and Castelnuovo-Mumford regularity. In Section 3, we prove Theorem \ref{main_thm} for $s=2$. In Section 4, we prove Theorem \ref{main_thm} for $s = 3$. Finally, Section 5 contains some applications of the main results.  

\section{Castelnuovo-Mumford regularity, symbolic powers and degree complexes}
In this section, we recall some definitions and properties concerning Castelnuovo-Mumford regularity, the symbolic powers of a squarefree monomial ideal, and the degree complexes of a monomial ideal. The interested reader is referred to (\cite{BH, D, E, S}) for more details.

\subsection{Graph theory}
Throughout this paper, $G$ will denote a finite simple graph over the vertex set $V(G)=[n] = \{1,2,\ldots,n\}$ and the edge set $E(G)$. For a vertex $x\in V(G)$, let the neighbour of $x$ be the subset $N_G(x)=\{y\in V(G)~|~ \{x,y\}\in E(G)\}$, and set $N_G[x]=N_G(x)\cup\{x\}$. For a subset $U$ of the vertices set $V(G)$, $N_G(U)$ and $N_G[U]$ are defined by $N_G(U)=\cup_{u\in U}N_G(u)$ and $N_G[U]=\cup_{u\in U}N_G[u]$. If $G$ is fixed, we will use $N(U)$ or $N[U]$ for short.

A subgraph $H$ is called an induced subgraph of $G$ if for any vertices $u,v\in V(H)\subseteq V(G)$ then $\{u,v\}\in E(H)$ if and only if $\{u,v\}\in E(G)$. For a subset $U$ of the vertices set $V(G)$, we shall denote by $G[U]$ the induced subgraph of $G$ on $U$, and denote by $G - U$ the induced subgraph of $G$ on $V(G)\setminus U$.  

A $m$-cycle in $G$ is a sequence of $m$ distinct vertices $1,\ldots, m\in V(G)$ such that $\{1,2\},\ldots, \{m-1,m\}, \{m,1\}$ are edges of $G$. We shall also use $C=12\ldots m$ to denote the $m$-cycle whose sequence of vertices is $1,\ldots, m$. 

\subsection{Simplicial complex} 
Let $\Delta$ be a simplicial complex on $[n]=\{1,\ldots, n\}$ that is a collection of subsets of $[n]$ closed under taking subsets. We put $\dim F = |F|-1$, where $|F|$ is the cardinality of $F$. The dimension of $\Delta$ is $\dim \Delta = \max \{ \dim F \mid F \in \Delta \}$.  It is clear that $\Delta$ can be uniquely  determined by the set of its maximal elements under inclusion, called by facets, which is denoted by $\F(\Delta)$.

A simplicial complex $\D$ is called a cone over $x\in [n]$ if $x\in B$ for any $B\in \F(\Delta)$. If $\D$ is a cone, then it is acyclic (i.e., has vanishing reduced homology).

For a face $F\in\Delta$, the link of $F$ in $\Delta$ is the subsimplicial complex of $\Delta$ defined by
$$\lk_{\Delta}F=\{G\in\Delta \mid  F\cup G\in\Delta, F\cap G=\emptyset\}.$$

\subsection{Stanley-Reisner correspondence}
Let $S = K[x_1,\ldots,x_n]$. We now recall the Stanley-Reisner correspondence which corresponds a squarefree monomial ideal of $S$ and a simplicial complex $\Delta$ on $[n]$. For each subset $F$ of $[n]$, let $x_F=\prod_{i\in F}x_i$ be a squarefree monomial in $S$. 

\begin{defn}For a squarefree monomial ideal $I$, the Stanley-Reisner complex of $I$ is defined by
$$ \Delta(I) = \{ F \subset [n] \mid x_F \notin I\}.$$

For a simplicial complex $\Delta$, the Stanley-Reisner ideal of $\Delta$ is defined by
$$I_\Delta = (x_F \mid  F \notin \Delta).$$
The Stanley-Reisner ring of $\Delta$ is the quotient by the Stanley-Reisner ideal, $K[\Delta] =  S/I_\Delta.$
\end{defn}
From the definition, it is easy to see the following:
\begin{lem}\label{cone} Let $I, J$ be squarefree monomial ideals of  $S = K[x_1,\ldots, x_n]$. Then 
\begin{enumerate}
    \item $\Delta(I)$ is a cone over $t \in [n]$ if and only if $x_t$ is not divided by any minimal generator of $I$.
    \item $\Delta(I + J) = \Delta(I) \cap \Delta(J).$
    \item $\Delta(I \cap J) = \Delta(I) \cup \Delta(J).$
\end{enumerate}
\end{lem}

\subsection{Castelnuovo-Mumford regularity} 
Let $\m = (x_1,\ldots, x_n)$ be the maximal homogeneous ideal of $S = K[x_1,\ldots, x_n]$ a polynomial ring over a field $K$. For a finitely generated graded $S$-module $L$, let
$$a_i(L)=
\begin{cases}
\max\{j\in\ZZ \mid H_{\m}^i(L)_j \ne 0\} &\text{ if  $H_{\m}^i(L)\ne 0$}\\ 
-\infty &\text{ otherwise,}
\end{cases}
$$
where $H^{i}_{\m}(L)$ denotes the $i$-th local cohomology module of $L$ with respect to $\m$. Then, the Castelnuovo-Mumford regularity (or regularity for short) of $L$ is defined to be
$$\reg(L) = \max\{a_i(L) +i\mid i = 0,\ldots, \dim L\}.$$
The regularity of $L$ can also be defined via the minimal graded free resolution. Assume that the minimal graded free resolution of $L$ is
$$0\longleftarrow L\longleftarrow F_0\longleftarrow F_1\longleftarrow\cdots\longleftarrow F_p\longleftarrow 0.$$
Let $t_i(L)$ be the maximal degree of graded generators of $F_i$. Then,
$$\reg(L) = \max\{t_i(L) - i\mid i = 0, \ldots, p\}.$$

From the minimal graded free resolution of $S/J$, we obtain $\reg(J)=\reg(S/J)+1$ for a non-zero and proper homogeneous ideal $J$ of $S$.

\subsection{Symbolic powers} 
Let $I$ be a non-zero and proper homogeneous ideal of $S$. Let $\{P_1,\ldots,P_r\}$ be the set of the minimal prime ideals of $I$. Given a positive integer $s$, the $s$-th symbolic power of $I$ is defined by
$$I^{(s)}=\bigcap_{i=1}^r I^sS_{P_i}\cap S.$$

For $f \in S$ and $x^\a = x_1^{a_1} ... x_n^{a_n}$, we denote 
$\frac{\partial (f)}{\partial (x^\a)}$ the partial derivative of $f$ with respect to $x^\a$. For each $s$, we denote 
$$I^{\langle s \rangle} =  ( f \in S ~|~ \frac{\partial f }{\partial x^\a} \in I, \text{ for all } x^\a \text{ with } |\a| = s-1  ), $$
the $s$-th differential power of $J$. When $K$ is a field of characteristic $0$ and $I$ is a radical ideal, it is a well-known theorem of Nagata-Zariski that $I^{(s)} = I^{\langle s \rangle}.$

When $f$ is a monomial, we denote $\frac{\partial^* (f)}{\partial^*(x^\a)}$ the $*$-partial derivative of $f$ with respect to $x^\a$, which is derivative without coefficients. In general, $\partial f / \partial x^\a = c \partial^*(f) / \partial^*(x^\a)$ for some constant $c$. Similarly, we define 
$$I^{[s]} =  ( f \in S ~|~ \frac{\partial^* f }{\partial^* x^\a} \in I, \text{ for all } x^\a \text{ with } |\a| = s -1),$$ the $s$-th $*$-differential power of $I$. When the characteristic of $K$ is equal to $0$, $I^{\langle s \rangle} = I^{[s]}.$ In general, we only have $I^{\langle s \rangle} \subseteq I^{[s]}.$ When $I$ is a squarefree monomial ideal, we first prove that the symbolic powers of $I$ is equal to the $*$-differential powers of $I$. 

\begin{lem}\label{differential_criterion} Let $I$ be a squarefree monomial ideal. Then $I^{(s)} = I^{[s]}.$
\end{lem}
\begin{proof} This result is folkloric, though we could not find a reference so we give a simple proof here. Let $P_1, \ldots, P_r$ be the minimal prime ideals of $I$. Then note that $I^{(s)} = P_1^s \cap \cdots \cap P_r^s$. Let $f$ be a monomial in $I^{(s)}$ and $\a \in \NN^n$ such that $|\a| = s-1$. Since $f\in P_i^s$ for all $i = 1, \ldots, r$, $\partial^* f /\partial^* x^\a \in P_i$. Thus, $\partial^*f /\partial^* x^\a \in P_1 \cap \cdots \cap P_r = I$. Therefore, $f \in I^{[s]}.$

Conversely, assume by contradiction that $I^{[s]}$ strictly contains $I^{(s)}$. Let $f= x_1^{r_1} ... x_n^{r_n}$ be a minimal generator of $I^{[s]}$ of smallest degree that is not in $I^{(s)}.$ In particular, $f \notin P^s$ for some minimal prime $P$ of $I$. Since $I$ is a squarefree monomial ideal, $P$ is generated by variables, without loss of generality, we assume that $P = (x_1, ..., x_t)$. Since $f \notin P^s$, $r_1 + ... + r_t < s$. Take $\a = (r_1, ..., r_t, s-1 - r_1 -... - r_t, 0, .., 0)$, then $|\a| = s-1$. Now, we have $\partial^*(f) / \partial^* (x^\a) ~|~ x_{t+1}^{r_{t+1}} ... x_n^{r_n}$, which is not contained in $P$. Thus, $\partial^*(f) / \partial^* (x^\a) \notin I$, a contradiction.
\end{proof}

For a monomial $f$ in $S$, we denote $\supp(f)$, the support of $f$, the set of all indices $i \in [n]$ such that $x_i|f$. For an exponent $\a \in \ZZ^n$, we denote $\supp(\a) =\{i\in [n] \mid \ a_i \neq 0\}$, the support of $\a$. For any subset $V \subset [n]$, we denote $$I_V = (f~|~  f \text{ is a monomial which belongs to } I \text{ and  } \supp (f) \subseteq V)$$ be the restriction of $I$ on $V$. We have

\begin{cor}\label{restriction} Let $I$ be a squarefree monomial ideal and $f$ be a monomial in $S$. Denote $V = \supp (f)$. Then, $f \in I^{(s)}$ if and only if $f \in I_V^{(s)}.$
\end{cor}
\begin{proof} Since $I_V \subseteq I$, if $f \in I_V^{(s)}$ then $f \in I^{(s)}.$ Conversely, assume that $f \in I^{(s)}.$ For any $\b$ such that $\supp \b \subseteq V$ and $|\b| = s-1$, by Lemma \ref{differential_criterion} $\partial^*(f) /\partial^*(x^{\b}) \in I$. But $\supp (\partial^*(f) /\partial^*(x^{\b})) \subseteq \supp (f) = V$, thus $\partial^*(f) /\partial^*(x^{\b}) \in I_V.$ By Lemma \ref{differential_criterion}, $f \in I_V^{(s)}.$
\end{proof}

As a consequence of Corollary \ref{restriction}, we deduce a generalization of \cite[Corollary 4.5]{GHOS} for squarefree monomial ideals.
\begin{cor}\label{restriction_inq} Let $I$ be a squarefree monomial ideal in $S$. Let $V \subseteq [n]$, and $I_V$ be the restriction of $I$ to $V$. Then for all $s \ge 1$,
$$\reg I_V^{(s)} \le \reg I^{(s)}.$$
\end{cor}
\begin{proof}
 By Corollary \ref{restriction}, $I_V^{(s)}$ is the restriction of $I^{(s)}$ to $V$. Let $\{t,\ldots,n\} = [n] \setminus V$. Then, $I_V^{(s)} + (x_t,...,x_n) = I^{(s)} + (x_t,...,x_n).$ The conclusion follows from \cite[Proposition 3.11]{NV2} and the fact that $x_t, ..., x_n$ is a regular sequence with respect to $S/I_V^{(s)}$.
\end{proof}

\subsection{Edge ideals and their symbolic powers} Let $G$ be a simple graph over the vertex set $V(G)=[n]=\{1,2,\ldots,n\}$. The edge ideal of $G$ is defined to be
$$I(G)=(x_ix_j~|~\{i,j\}\in E(G))\subseteq S.$$
For simplicity, we often write $i \in G$ (resp. $ij \in G$) instead of $i \in V(G)$ (resp. $\{i,j\} \in E(G)$).

A clique of size $t$ in $G$ is an induced subgraph of $G$ which is a complete graph over $t$-vertices. We also called a clique of size $3$ a triangle. 

Let $J_1(G)$ be the ideal generated by all squarefree monomials $x_ix_jx_r$ where $\{i,j,r\}$ forms a triangle in $G$. Let $J_2(G)$ be the ideal generated by all squarefree monomials $x_ix_jx_rx_s$ where $\{i,j,r,s\}$ forms a clique of size $4$ in $G$ and all squarefree monomials $x_C$ where $C$ is a $5$-cycle of $G$. 

We have the following expansion formula of the second and third symbolic powers of an edge ideal. Note that the first formula is \cite[Corollary 3.12]{Su}.
\begin{thm}\label{expansion_second} Let $I$ be the edge ideal of a simple graph $G$. Then
$$I^{(2)} = I^2 + J_1(G).$$
\end{thm}
\begin{proof} Using Lemma \ref{differential_criterion}, it is easy to see that the left hand side contains the right hand side. Conversely, let $f \in I^{(2)}$ be a monomial generator. By Lemma \ref{restriction}, we may assume that $\supp f = V(G)$. If $G$ contains a triangle, then $f \in J_1(G)$. If $G$ contains a cycle of length $ \ge 4$, then $f \in I^2$. Thus we may assume that $G$ do not contain any cycles. In this case $I(G)^{(2)} = I(G)^2$, thus $f \in I^2$. 
\end{proof}
\begin{thm}\label{expansion_third} Let $I$ be the edge ideal of a simple graph $G$. Then
$$I^{(3)} = I^3 + I J_1(G) + J_2(G).$$
\end{thm}
\begin{proof}
Using Lemma \ref{differential_criterion}, it is easy to see that the left hand side contains the right hand side. Conversely, let $f \in I^{(3)}$ be a monomial generator. By Lemma \ref{restriction}, we can assume that $\supp (f) = V(G)$. If $G$ contains a $5$-cycle, then $f \in J_2(G)$. If the matching number of $G$ is at least 3, then $f \in I^3$. Thus, we may assume that $G$ does not contain a cycle of length $\ge 5$. If $G$ does not contain a triangle, then $G$ is bipartite, and thus $I(G)^{(3)} = I(G)^3$, and we are done. Thus we may assume that $G$ contains a triangle, called $123$. Let $f = x_1^{\alpha_1}\cdots x_n^{\alpha_n}$ where $\alpha_i \ge 1$. If two of the exponents $\alpha_1, \alpha_2, \alpha_3$ is at least $2$, then $f \in IJ_1(G)$. Thus, assume that $\alpha_2 = \alpha_3 = 1$. By Lemma \ref{differential_criterion}, $\partial^*f / \partial^*(x_2x_3) = x_1^{\alpha_1} x_4^{\alpha_4} \cdots x_n^{\alpha_n} \in I$. If $x_4^{\alpha_4} \cdots x_n^{\alpha_n} \in I$ then $f \in IJ_1(G)$. Thus, we may assume that $x_4^{\alpha_4} \cdots x_n^{\alpha_n} \notin I$, and $x_1x_4 \in I$. If $\alpha_1 > 1$, then $f \in IJ_1(G)$ as $(x_1x_4)\cdot(x_1x_2x_3) | f$. Thus, we may assume that $\alpha_1 = 1$. Similarly, we deduce that $x_2 x_i$ and $x_3 x_j \in I$ for some $i, j \ge 4$. If $|\{4,i,j\}| = 3$, then $f \in I^3$. If $|\{4,i,j\}| = 1$, then $f \in J_2(G)$, as $\{1,2,3,4\}$ forms a clique of size $4$. Now, assume that $i =4$, and $j \neq 4$. In this case $x_1x_2x_4 \in J_1(G)$ and thus $(x_1x_2x_4)(x_3x_j) \in IJ_1$. This concludes our proof.
\end{proof}

\subsection{Degree complexes}
For a monomial ideal $I$ in $S$, Takayama in \cite{T} found a combinatorial formula for $\dim_KH_\m^i(S/I)_\a$ for all $\a\in\ZZ^n$ in terms of certain simplicial complexes which are called degree complexes. For every $\a = (a_1,\ldots, a_n) \in \ZZ^n$ we set $G_\a = \{i\mid \ a_i < 0\}$ and write $x^{\a} = \Pi_{j=1}^n x_j^{a_j}$. Thus, $G_\a =\emptyset$ whenever $\a \in \NN^n$. The degree complex $\D_\a(I)$ is the simplicial complex whose faces are $F \setminus G_\a$, where $G_\a\subseteq F\subseteq [n]$, so that for every minimal generator $x^\b$ of $I$ there exists an index $i \not\in F$ with $a_i < b_i$. It is noted that $\D_\a(I)$ may be either the empty set or $\{\emptyset\}$ and its vertex set may be a proper subset of $[n]$. The next lemma is useful to compute the regularity of a monomial ideal in terms of its degree complexes.

\begin{lem}\label{Key0}
Let $I$ be a monomial ideal in $S$. Then
\begin{multline*}
\reg(S/I)=\max\{|\a|+i~|~\a\in\NN^n,i\ge 0,\h_{i-1}(\lk_{\D_\a(I)}F;K)\ne 0\\ \text{ for some $F\in \D_\a(I)$ with $F\cap \supp \a=\emptyset$}\}.
\end{multline*}
In particular, if $I=I_\D$ is the Stanley-Reisner ideal of a simplicial complex $\D$ then
$\reg(K[\D])=\max\{i~|~i\ge 0,\h_{i-1}(\lk_{\D}F;K)\ne 0\text{ for some }F\in \D\}$.
\end{lem}
\begin{proof}
By \cite[Proposition 2.5]{CHHKTT},
$$\reg(S/I)=\max\{|\a|+|G_\a|+i~|~\a\in\ZZ^n,i\ge 0,\h_{i-1}(\D_\a(I);K)\ne 0\}.$$
Assume, $\reg(S/I)=|\a|+|G_\a|+i$ for some $\a\in \ZZ^n$. By maximality, we can assume $\a_i=-1$ for all $i\in G_{\a}$. Let ${\a}^+\in \NN^n$ by ${\a}^{+}_{i}=a_i$ if $i\notin G_\a$ and ${\a}^{+}_{i}=0$ otherwise. Then, $\D_\a(I)=\lk_{\D_{\a^+}(I)}G_\a$; $\reg(S/I)=|{\a}^+|+i$ and $\h_{i-1}(\lk_{\D_{\a^+}(I)}G_\a;k)\ne 0$. It implies that 
\begin{multline*}
\reg(S/I)\le\max\{|\a|+i~|~\a\in\NN^n,i\ge 0,\h_{i-1}(\lk_{\D_\a(I)}F;K)\ne 0 \\
\text{ for some $F\in \D_\a(I)$ with $F\cap \supp \a=\emptyset$}\}.
\end{multline*}
Converserly, if $\h_{i-1}(\lk_{\D_\a(I)}F;K)\ne 0$ for some $\a\in\NN^n,i\ge 0, F\in \D_\a(I)$ with $F\cap \supp \a=\emptyset$, we only put $\b\in\ZZ^n$ such that $\b_i=-1$ for all $i\in F$ and $b_i=a_i$ otherwise. Then, $\h_{i-1}(\D_\b(I);K)\ne 0$ and $|\b|+|G_\b|+i=|\a|+i$ and using again \cite[Proposition 2.5]{CHHKTT}, we obtain the converse inequality. 

If $I=I_\D$ is the Stanley-Reisner ideal of a simplicial complex $\D$, as in the proof of \cite[Theorem 1]{T}, $\tilde H_i(\D_\a(I),K) = 0$ for all $i$ for each $\a\in\NN^n$ such that there is a component $a_j \ge 1$. Then, we only consider $a_j=0$ for all $j$ (i.e. $\a=\o$). From definition, $\D_\o(I)=\D$. This completes our proof.
\end{proof}

\begin{rem} Let $I$ be a monomial ideal in $S$ and a vector $\a\in \ZZ^n$. In the proof of Theorem 1 in \cite{T}, he showed that if there exists $j\in [n]\setminus G_\a$ such that $a_j\ge \rho_j=\max\{\deg_{x_j}(u)~|~ u\text{ is a minimal monomial generator of } I\}$ then $\D_\a(I)$ is a cone. Thus, we only consider some vectors $\a$ which belongs to the finite set
	$$\Gamma(I)=\{\a\in\NN^n~|~ a_j<\rho_j\text{ for all } j=1,\ldots,n\}.$$ 
\end{rem}

From this, we obtain an upper bound of the regularity of a monomial ideal in terms of its degree complexes.

\begin{cor}
Let $I$ be a monomial ideal in $S$. Then
$$\reg(S/I)\le \max\{|\a|+\reg(K[\D_\a(I)])~|~\a\in\Gamma(I)\}.$$
\end{cor}

One might expect that this inequality becomes an equality. Unfortunately, this is not the case.

\begin{exm} Let $I = x_1 (x_2,x_3,x_4,x_5)$ be an edge ideal in $S=K[x_1,\ldots,x_{5}]$. For each $s \ge 2$, let $x^\a = (x_2x_3x_4x_5)^{s-1}.$ Then, $\a\in\Gamma(I^s)$. Furthermore, we have $\reg(S/I^s)=2s-1$, $|\a| = 4s-4$, and $\reg(K[\D_\a(I^s))]=0$. Thus, $\reg(S/I^s)<\max\{|\a|+\reg(K[\D_\a(I^s)])~|~\a\in\Gamma(I^s)\}.$
\end{exm}

We call a pair $(\a,i) \in \NN^n\times\NN$ {\it an extremal exponent of the ideal $I$}, if $\reg(S/I) = |\a| + i$ as in Lemma \ref{Key0}. It is clear that if $(\a,i) \in \NN^n\times\NN$ is an extremal exponent of $I$ then $x^\a \notin I$ and $\Delta_{\a}(I)$ is not a cone over $t$ with $t\in\supp\a$. From the definition, it is easy to see the following

\begin{lem}\label{extremal_red} Let $I, J$ be proper monomial ideals of $S$. Let $(\a,i)$ be an extremal exponent of $I$. If $\Delta_\a(I) = \Delta_\a(J)$, then $\reg I \le \reg J$. In particular, if $J \subseteq I$ and $\Delta_\a(I) = \Delta_\a(J)$ for all exponent $\a\in\NN^n$ such that $x^\a \notin I$ then $\reg I \le \reg J.$
\end{lem}
\begin{proof} By definition, there exists a face $F \in \Delta_\a(I)$ such that $ \supp F \cap \supp \a = \emptyset$ and $\h_{i-1}(\lk_{\Delta_\a(I)} F;K) \neq 0$ and $\reg I = |\a| + i$. Since $\Delta_\a(I) = \Delta_\a(J)$, by Lemma \ref{Key0} $\reg J \ge |\a| + i$ as required.
\end{proof}

The next lemma appeared in \cite{MTru}, and we would like to sketch the proof for completeness. It is very useful to compute the degree complexes of powers of an edge ideal in this paper.

\begin{lem}\label{Key1}
Let $I$ be a monomial ideal in $S$ and $\a\in\NN^n$. Then
$$I_{\Delta_{\a}(I)}=\sqrt{I : x^\a}.$$
In particular, $x^\a\in I$ if and only if $\Delta_{\a}(I)$ is the void complex. 
\end{lem}
\begin{proof}
For any $F\subseteq [n]$, let $x_F=\prod_{i\in F}x_i$. We have
\begin{multline*}
x_F\in I_{\Delta_{\a}(I)}\Longleftrightarrow F\notin\Delta_{\a}(I) \Longleftrightarrow \exists ~ x^\b\in G(I)\text{ such that }\forall~ i\notin F, b_i\le a_i\\
\Longleftrightarrow \exists~ t\in\NN\setminus\{0\}, (x_F)^tx^\a\in I \Longleftrightarrow x_F\in\sqrt{I : x^\a}.
\end{multline*}
\end{proof}

The following lemma is essential to using the induction method in studying the regularity of a monomial ideal.

\begin{lem}\label{red0} Let $I$ be a monomial ideal and the pair $(\a,i) \in \NN^n\times\NN$ be its extremal exponent. If $x$ is a variable that appears in $\sqrt{I:x^\a}$ and $x \notin \supp \a$, then $$\reg (I) = \reg (I,x).$$
\end{lem}
\begin{proof}
    Let $J = (I,x)$. Then, we know that $\reg J \le \reg I$ (see, for example \cite[Proposition 3.11]{NV2}). Thus, it suffices to prove the reverse inequalities. As $x$ does not belong to support of $\a$, then
    $$\sqrt{J:x^\a}=\sqrt{I:x^\a}+(x)=\sqrt{I:x^\a}.$$
By Lemma \ref{Key1}, $\Delta_\a(I) = \Delta_\a(J)$. The conclusion follows from Lemma \ref{Key0}. 
\end{proof}

Since we will deal with radical ideals of the colon ideals of monomial ideals, the following simple observation will be useful later on.

\begin{lem}\label{radical_colon} Let $I$ be a monomial ideal in $S$ generated by the monomials $f_1, ..., f_s$ and $\a \in \NN^n$. Then $\sqrt{I:x^\a}$ is generated by $\sqrt{f_1/\gcd(f_1, x^\a)}, ..., \sqrt{f_s/\gcd(f_s,x^\a)}$, where $\sqrt{x^\b}=\prod_{i\in\supp \b}x_i$ for each $\b\in\NN^n$.
\end{lem}
\begin{proof} Let $g$ be a minimal generator of $\sqrt{I:x^\a}$. Then there exists a natural number $t >0$ such that $g^tx^\a \in I.$ We may assume that $f_1 | g^tx^\a.$ In particular, $f_1/\gcd(f_1,x^\a) |g^t.$ Taking radical, we deduce that $\sqrt{f_1/\gcd(f_1,x^\a)} | g$. This concludes our proof.
\end{proof}

\section{Proof of Theorem \ref{main_thm} for $s = 2$}\label{sec_2}

Let $G$ be a simple graph with vertex set $[n]$ and edge set $E(G)$. Let $I=I(G)$ be the edge ideal of $G$. In this section, we will prove Theorem \ref{main_thm} for $s = 2$. First, we give a property of the degree complexes of the second symbolic/ordinary power of $I$. 
 
\begin{lem}\label{Key2}
Let $I = I(G)$ and $\a\in\NN^n$ such that $x^\a\notin I^{(2)}$. Then, 
$$\sqrt{I^{(2)}:x^\a}=\sqrt{I^2:x^\a}.$$
In particular, $\D_\a(I^{(2)})=\D_\a(I^{2})$. 
\end{lem}
\begin{proof} 
By Theorem \ref{expansion_second}, it suffices to prove that if $f$ is a minimal squarefree monomial generator of $\sqrt{J_1(G):x^\a}$ then $f \in \sqrt{I^2:x^\a}$. By Lemma \ref{radical_colon}, we may assume that $f = x_1x_2x_3 / \gcd(x_1x_2x_3,x^\a)$, where $123$ is a triangle in $G$. Since $x^\a\notin I^{(2)}$, $\deg f \ge 1$. There are two cases. 

\smallskip
\noindent\textbf{Case 1:} $\deg f=1$. May assume that $f = x_1$. Thus $x_2x_3 | x^\a$. Now, $x_1^2x_2x_3\in I^2$, which implies that $f\in \sqrt{I^2:x^\a}$.

\smallskip
\noindent\textbf{Case 2:} $\deg f \ge 2$. Therefore, one of three monomials $x_1x_2,x_2x_3,x_1x_3$ will be a divisor of $f$. In particular, $f \in I\subseteq \sqrt{I^2:x^\a}$. 

By Lemma \ref{Key1}, we obtain $\D_\a(I^{(2)})=\D_\a(I^{2})$ for any $\a\in\NN^n$ such that $x^\a\notin I^{(2)}$.  
\end{proof}

We are now in a position to prove the main result of this section.

\begin{thm}\label{Main1} Let $G$ be a simple graph and $I$ its edge ideal. Then 
$$\reg(I^{(2)}) = \reg(I^2).$$
\end{thm}
\begin{proof} By Lemma \ref{extremal_red} and Lemma \ref{Key0}, $\reg I^{(2)} \le \reg I^2.$ 

Conversely, we prove by induction on $n = |V(G)|$ that $\reg(S/I^2) \le \reg (S/I^{(2)})$. By results of \cite{HNTT,NV2}, it suffices to consider the case $G$ is connected and has no isolated vertices. Moreover, one can see that if $n\le 3$ then $I$ is either $xy$ or $(xy,yz)$ or $(xy,yz,zx)$ which satisfy our assertion. Put $n\ge 4$. 

Let $(\a,i)\in\NN^n\times\NN$ be an extremal exponent of $I^2$. If $x^\a \notin I^{(2)}$, by Lemma \ref{Key2}, $\D_\a(I^{2})=\D_\a(I^{(2)})$. By Lemma \ref{extremal_red}, $\reg I^2 \le \reg (S/I^{(2)})$. If $x^\a$ is divisible by a triangle, say $x_1x_2x_3$ i.e. $12,23,13\in E(G)$.
Since $G$ is connected and $n\ge 4$, we have $N(\{1,2,3\})\ne \emptyset$. Let $r\in N(\{1,2,3\})$. One can see that $r\notin \supp \a$ by $x^\a \notin I^2$. It is clear that $x_r\in\sqrt{I^2:x^\a}$. Using Lemma \ref{red0}, we have $$\reg(I^2)=\reg(I^2,x_r)=\reg((I,x_r)^2,x_r).$$
Moreover, $\reg((I,x_r)^2,x_r)=\reg(J^2,x_r)$, where $J$ is the edge ideal of the restriction of $G$ to $V \setminus\{r\}$. It is noted that $x_r$ is a regular element of $S/J^2$ then $\reg(J^2,x_r)=\reg(J^2)$. By induction, $\reg(S/J^{2})\le \reg(S/J^{(2)})$. By Corollary \ref{restriction_inq}, $\reg(S/J^{(2)})\le \reg(S/I^{(2)})$ as required.
\end{proof}

\section{Proof of Theorem \ref{main_thm} for $s = 3$}\label{sec_3}

Let $G$ be a simple graph with vertex set $[n]$ and edge set $E(G)$. Let $I=I(G)$ be the edge ideal of $G$. In this section, we will prove Theorem \ref{main_thm} for $s = 3$. First, we prove a technical lemma for degree complexes of the third symbolic/ordinary power of $I$.

\begin{lem}\label{Key3} Let $\a \in \NN^n$ such that $x^\a \notin I^{(3)}$. Assume that $\sqrt{I^{(3)}:x^\a} \neq \sqrt{I^3:x^\a}$. Let $f$ be a minimal squarefree monomial generator $\sqrt{I^{(3)} : x^\a}$ such that $f \notin \sqrt{I^3:x^\a}$. Then we have
\begin{enumerate}
    \item There exists a triangle $123$ in $G$ such that $x_1x_2x_3 | x^\a.$
    \item $\deg f = 1$ and $\supp f \notin \supp \a$.
\end{enumerate}
\end{lem}
\begin{proof} By Theorem \ref{expansion_third} and Lemma \ref{radical_colon}, there are three cases as follows.

\smallskip
\noindent\textbf{Case 1:} There exists a clique of size $4$, $C=1234$, of $G$ such that 
$$f  = \sqrt{x_1 x_2x_3x_4 / \gcd(x_1x_2x_3x_4,x^\a)}.$$
If $\deg f \ge 2$ then $\supp (f)$ must contain at least two vertices $i,j$ among $\supp C$. In particular, $f \in I\subseteq \sqrt{I^3:x^\a},$ which is a contradiction. If $\deg f = 1$, say $f = x_1$. This implies that $x_2x_3x_4 | x^\a$. But, $f^3x^\a\in I^3$ by $x_1^3(x_2x_3x_4)=(x_1x_2)(x_1x_3)(x_1x_4)\in I^3$, which is a contradiction.  

\smallskip
\noindent\textbf{Case 2:} There exists a $5$-cycle, $C = 12345$, of $G$ such that 
$$f = \sqrt{x_1x_2x_3x_4x_5/ \gcd(x_1x_2x_3x_4x_5,x^\a)}.$$
Since $f \notin \sqrt{I^3:x^\a}$, $f \notin I$. Furthermore, $f | x_1x_2x_3x_4x_5$, we have three subcases.

Subcase 2.1. $\deg f = 3$. We may assume that $f = x_1x_3x_5$ then $x_2x_4 | x^\a$. In this case $f^2 x_2x_4 \in I^3$, which is a contradiction.

Subcase 2.2. $\deg f = 2$. We may assume that $f = x_1x_3$ then $x_2x_4x_5|x^\a$. In this case $f^2 x_2x_4x_5 \in I^3$, which is a contradiction.

Subcase 2.3 $\deg f = 1$. We may assume that $f = x_1$ then $x_2x_3x_4x_5|x^\a$. In this case $f^2 x_2x_3x_4x_5 \in I^3$, which is a contradiction. 

\smallskip
\noindent\textbf{Case 3:} There exists an edge $uv$ and a triangle $123$ in $G$ such that $$f=\sqrt{x_ux_v x_1x_2x_3/\gcd(x_ux_vx_1x_2x_3,x^\a)},$$
note that $u,v$ might belong to $\{1,2,3\}.$ In particular, $\supp (f) \subseteq \{u,v\} \cup \{1,2,3\}$. Since $f\notin \sqrt{I^3:x^\a}$, $f\notin I$. In particular, $|\supp (f) \cap \{u,v\}| \le 1$ and $|\supp (f) \cap \{1,2,3\} | \le 1.$ There are two subcases.

Subcase 3.1. $\deg f \ge 2$. Since $f$ is squarefree, $|\supp (f) \cap \{u,v\}| = 1$ and $|\supp (f) \cap \{1,2,3\}| = 1$. We may assume that $f = x_1x_u$. In particular $x_vx_2x_3$ is a divisor of $x^\a$. In this case, $f^2x^\a \in I^3$ by $ x_u^2x_1^2x_vx_2x_3=x_u(x_ux_v)(x_1x_2)(x_1x_3)\in I^3$, which is a contradiction.

Subcase 3.2. $\deg f = 1$. We first prove that $\supp f\notin \{1,2,3\}$. Assume by contradiction that $\supp f \in \{1,2,3\}$. We may assume that $f = x_1$. 
\begin{enumerate} 
\item If $x_1\in\{x_u,x_v\}$. Assume $x_1=x_u$. It implies that $x_vx_2x_3$ is a divisor of $x^\a$. By $(x_1x_v)(x_1x_2)(x_1x_3)\in I^3$, we have $x_1\in \sqrt{I^3:x^\a}$.
\item If $x_1\notin\{x_u,x_v\}$. Then $x_ux_vx_2x_3$ is a divisor of $x^\a$. Since $x_1^2(x_ux_vx_2x_3)=(x_ux_v)(x_1x_2)(x_1x_3)\in I^3$, $x_1\in \sqrt{I^3:x^\a}$.
\end{enumerate}
This is a contradiction. Thus $\supp f \notin \{1,2,3\}$. Therefore, $x_1x_2x_3 | x^\a$. Furthermore, $\supp f \in \{u,v\}$. We may assume that $f = x_u$. If $\supp f \in \supp \a$, then $x_u^2 | x_ux_vx_1x_2x_3$. Thus $u = \supp f \in \{1,2,3\}$, which is a contradiction. This completes our prooof.
\end{proof}

\begin{exm} One might hope that in general we have for $\a \in \NN^n$ and $x^\a \notin I^{(s)}$ then 
$$\sqrt{I^{(s)} : x^\a} = \sqrt{I^s : x^\a} + (\text{variables}).$$
Unfortunately, this is not the case for $s \ge 4.$ Indeed, let 
$$I = (x_1x_2,x_2x_3,x_3x_1,x_1x_4,x_4x_5,x_2x_6,x_6x_7)\text{ and }x^\a = x_1x_2x_3x_4x_6,$$
 then $x_5x_7$ is a minimal generator of $\sqrt{I^{(4)} : x^\a}$ but does not belong to $\sqrt{I^4:x^\a}$.
\end{exm}

We are now in a position to prove the first inequality of the main result of this section.

\begin{thm}\label{thirdpower_in} Let $G$ be a simple graph and $I=I(G)$. Then
$$\reg(I^{(3)})\le \reg(I^3).$$
\end{thm}
\begin{proof} We prove by induction on $n = |V(G)|$. Let $(\a,i)$ be an extremal exponent of $I^{(3)}$. By Lemma \ref{extremal_red}, we may assume that $\Delta_\a(I^{(3)}) \neq \Delta_\a(I^3)$. By Lemma \ref{Key1} and Lemma \ref{Key3}, there exists a variable, $x_t \in \sqrt{I^{(3)}:x^\a}$ such that $x_t \notin \supp x^\a$. By Lemma \ref{red0}, $\reg I^{(3)} = \reg (I^{(3)},x_t).$ Let $J$ be the restriction of $I$ to $V(G) \setminus \{t\}$. Then by induction $\reg J^{(3)} \le \reg J^3$. Thus,
$$\reg I^{(3)} = \reg (J^{(3)},x_t) \le \reg (J^3,x_t) \le \reg I^3,$$
where the last inequality follows from \cite[Proposition 3.11]{NV2}.
\end{proof}

To prove the reverse inequality $\reg(S/I^3) \le \reg (S/I^{(3)})$, we also use induction on $n=|V(G)|$. By results of \cite{HNTT,NV2}, it suffices to consider the case $G$ is connected and has no isolated vertices. 

Moreover, one can see that if $n\le 3$ then $I$ is either $xy$ or $(xy,yz)$ or $(xy,yz,zx)$ which satisfy our assertion.  

For simplicity of exposition, throughout the rest of this section, we always assume that $(\a,i)\in\NN^n\times\NN$ is an extremal exponent of $I^3$ and $n\ge 4$, where $I=I(G)$. It is clear that $x^\a\notin I^3$. Then, we will fix a face $F\in \D_{\a}(I^3)$ such that $F\cap\supp(\a)=\emptyset$ and $\h_{i-1}(\lk_{\D_{\a}(I^3)} F;K)\ne 0$. By Lemma \ref{extremal_red}, it suffices to consider the cases $x^\a\notin I^{(3)}$ with $\D_\a(I^{(3)})\ne\D_\a(I^{3})$ or $x^\a\in I^{(3)}$.

We will need a series of lemmas for each form of $\a$. 

\begin{lem}\label{red1} If $x^\a\in J_2(G)$ and it is divisible by $x_C$, where $C$ is a $5$-cycle of $G$, then $\reg I^3 \le\reg I^{(3)}.$
\end{lem}
\begin{proof} Without loss of generality, we may assume that $x^\a = x_1 x_2 x_3 x_4 x_5\cdot f$, where $C=12345$ is a $5$-cycle of $G$ and a monomial $f\in S$. Since $x^\a\notin I^3$, $\supp (f)\cap N(C)=\emptyset$ and $\supp(f)$ is an independent set of $G$. If $n\ge 6$, by the connected property of $G$,  there exists a neighbor $r$ of $C$  and it does not belong to $\supp\a$. In particular, $x_r\in \sqrt{I^3:x^\a}$. By Lemma \ref{red0}, we have
    $$\reg I^3 = \reg (I^3,x_r) = \reg ((I,x_r)^3,x_r).$$
Similarly to the proof of Theorem \ref{Main1}, by induction and Corollary \ref{restriction_inq}, we have
            $$\reg ((I,x_r)^3,x_r)=\reg(J^3,x_r)=\reg(J^3)\le \reg (J^{(3)})\le \reg I^{(3)}.$$ 

If $n = 5$, then $x^\a = x_1x_2x_3x_4x_5$ and $\Delta_a(I^3) = \{\emptyset\}$. In particular, $ i = 0$, thus $\reg S/I^3 = |\a| + i = 5 \le \reg S/I^{(3)}.$
\end{proof}

In the next three lemma, we will use the following claim which is easy to see from Lemma \ref{cone}.

\smallskip
\noindent\textbf{Claim A:} Let $G$ be a connected simple graph and $I$ be its edge ideal. Let $r\in V(G)$ and $J$ be a squarefree monomial ideal such that $x_r$ does not appear in any of its minimal squarefree monomial generators. If $x_u\in J$ for all $u\in N_G(r)$ then the simplicial complex of $I+J$ is a cone over $r$.

\begin{lem}\label{red2} If $x^\a\in IJ_1(G)$ then $\reg(I^3) \le \reg(I^{(3)}).$
\end{lem}
\begin{proof} Without loss of generality, we may assume that $x^\a = x_1 x_2 x_3x_ux_v\cdot f$, where $C=123$ is a $3$-cycle; $uv$ is an edge of $G$ and a monomial $f\in S$. Since $x^\a\notin I^3$, $\supp(f)\cap N(\{1,2,3\})=\emptyset$.
		
We next distinguish some cases:

\smallskip
\noindent\textbf{Case 1:} $|N(\{1,2,3\})\setminus\{u,v,1,2,3\}|\ge 1$. Let $r\in N(\{1,2,3\})\setminus \{1,2,3,u,v\}$. So $r\notin \supp(f)$. Then $r\notin\supp\a$ and $x_r\in\sqrt{I^3:x^\a}$. With the same lines of the proof of Lemma \ref{red1},  we have $\reg(I^3) \le \reg(I^{(3)}).$

\smallskip
\noindent\textbf{Case 2:} $|N(\{1,2,3\})\setminus\{u,v,1,2,3\}|= 0$. By $n\ge 4$ and $G$ is connected, we must have $|\{u,v\}\setminus\{1,2,3\}|\ge 1$ and $|\{u,v\}\cap N(\{1,2,3\})|\ge 1$. If $v\in N(\{1,2,3\})$ and $u\notin N(\{1,2,3\})$. Hence $N(\{1,2,3\})=\{1,2,3,v\}$. By $x^\a\notin I^3$, we must have $\supp(f)\cup\{u\}$ is an independent set of $G$ (so $u\notin N(\supp(f))$) and $\supp(f)\cap N(\{1,2,3,u\})=\emptyset$. It is noted that $u$ may belong to $\supp(f)$.
One can check that
$$\sqrt{I^3:x^\a}=I+(x_i~|~i\in N(\{1,2,3,u\}\cup\supp(f))).$$
Hence, $x_u\notin \sqrt{I^3:x^\a}$. Using Claim A, $\D_\a(I^3)$ is a cone over $u\in\supp(\a)$, a contradiction.

Therefore, $u,v\in N(\{1,2,3\})$ i.e. $N(\{1,2,3\})=\{u,v,1,2,3\}$. If $V(G)$ properly contains $\{1,2,3,u,v\}$, then by the connectedness property, there exists $t\in N(\{1,2,3,u,v\})\setminus\{1,2,3,u,v\}$. So, $t\notin\supp(f)$ i.e. $t\notin\supp(\a)$. Now $x_t \in \sqrt{I^3:x^\a}$ by $x_1x_2x_3x_ux_vx_t\in I^3$. By Lemma \ref{red0}, and induction, we have $\reg(I^3) \le \reg(I^{(3)})$. Finally, if $V(G) = \{1,2,3,u,v\}$, then $\Delta_\a(I^3) = \{\emptyset\}$. Thus, $i = 0$, and $\reg S/I^3 = |\a| + i = 5 \le \reg S/I^{(3)}.$
\end{proof}

\begin{lem}\label{red3}  If $x^\a\in J_2(G)$ and it is divisible by $x_C$, where $C$ is a clique of size $4$ in $G$, then $\reg(I^3) \le \reg(I^{(3)}).$
\end{lem}
\begin{proof} Without loss of generality, we may assume that $x^\a = x_1 x_2 x_3x_4\cdot f$, where $1234$ is a clique of size $4$ in $G$ and a monomial $f\in S$. Using Lemma \ref{red2}, we may assume that $\supp(f)\cap N(\{1,2,3,4\})=\emptyset$. By $x^\a\notin I^3$, we must have $\supp(f)$ is either an independent set of $G$ or the empty set. If there exists $r\in N(1)\cap N(2)\setminus \{3,4\}$ then $x_r\in \sqrt{I^3:x^\a}$ and $r\notin\supp(\a)$. Using Lemma \ref{red0} and similarly with the proof of Theorem \ref{Main1}, we obtain $\reg(I^3) \le \reg(I^{(3)})$. Hence, we can assume that $N(i)\cap N(j)=\{1,2,3,4\}\setminus\{i,j\}$ for all $1\le i\ne j\le 4$. Then, one can see that
	\begin{multline*}
\sqrt{I^3:x^\a} = I+\sum_{1\le i\ne j\le 4} N(i)\cdot N(j)+(x_i~|~i\in N(\supp(f))\cup\{1,2,3,4\})
	\end{multline*}
	where $N(U)\cdot N(W)=(x_ix_j~|~i\in N(U), j\in N(W))$ is an ideal in $S$ for any $U, W\subset [n]$.
	
	If $f\ne 1$, we assume that $r\in \supp(f)$. Using again the Claim A, $\D_\a(I^3)$ is a cone over $r\in\supp(\a)$, a contradiction.
	
	If $f=1$, then $x^\a=x_1 x_2 x_3x_4$. Then, 
\begin{align*}
\sqrt{I^3:x^\a} &= I+\sum_{1\le i\ne j\le 4} N(i)\cdot N(j)+(x_1,x_2,x_3,x_4)\\
&= I + N(\{1,2,3\}) \cdot N(4) + N(\{1,2\})\cdot N(3) + N(1)\cdot N(2)+(x_1,x_2,x_3,x_4).
\end{align*}
Note that $N(\{1,2,3\}) = N[\{1,2,3\}]$ as $123$ is a triangle in $G$. Let $L=I + N(\{1,2\})\cdot N(3) + N(1)\cdot N(2)+(x_1,x_2,x_3,x_4)$. Then,  
$$\sqrt{I^3:x^\a}  = (L + (x_i~|~i\in N(\{1,2,3\}))) \cap (L + (x_i~|~ i\in N(4))).$$
Let $\D=\D_\a(I^3)$ and $\Gamma_1,\Gamma_2$ be the simplicial complexes which are corresponding to $L + (x_i~|~i\in N(\{1,2,3\}))$ and $L + (x_i~|~ i\in N(4))$. Hence, $$\Delta = \Gamma_1 \cup \Gamma_2.$$
Moreover, $\lk_\D(F)=\lk_{\Gamma_1} F\cup \lk_{\Gamma_2}F$. Applying the Mayer-Vietoris sequence, we have
\begin{multline*}
\cdots\rightarrow\h_{i-1}(\lk_{\Gamma_1}F;K) \oplus \h_{i-1}(\lk_{\Gamma_2}F;K) \rightarrow \h_{i-1}(\lk_{\Delta}F;K) \rightarrow \\
\h_{i-2}(\lk_{\Gamma_1}F \cap\;\lk_{\Gamma_2}F;K)\rightarrow\cdots
\end{multline*}
Since the middle term is nonzero, this implies that either the term on the left or the term on the right is nonzero. In particular, we have three cases.

\noindent\textbf{Case 1:} $\h_{i-2}(\lk_{\Gamma_1} F \cap \lk_{\Gamma_2} F;K) \ne 0$.  One can see that $\lk_{\Gamma_1} F \cap \lk_{\Gamma_2} F = \lk_{\Gamma_1\cap\Gamma_2} F$ and $\Gamma_1\cap\Gamma_2 = \Gamma$, which is the simplicial complex of $I+(x_i~|~ i\in N(\{1,2,3,4\}))$. It is noted that we also have 
$$\sqrt{I^3:x^\b}=I+(x_i~|~ i\in N(\{1,2,3,4\})),$$
where $x^\b=(x_1x_2)(x_1x_3x_4)$. Thus, $|\b|+i-1\le \reg(S/I^3)=|\a|+i=|\b|+i-1$. It implies that $(\b;i-1)$ is also an extremal exponent of $I^3$. This statement is done from the previous Lemma \ref{red2}.
 
\noindent\textbf{Case 2:} $\h_{i-1}(\lk_{\Gamma_1} F;K) \ne 0$. One can see that
$$\sqrt{I^3:(x_1^2x_2^2x_3)}=I+(x_i~|~ i\in N(\{1,2,3\}))=L + (x_i~|~i\in N(\{1,2,3\}))).$$
It means that $\Gamma_1=\D_{x_1^2x_2^2x_3}(I^3)$. Then, $\reg(S/I^3)\ge 5+i>|\a|+i$, a contradiction.

\noindent\textbf{Case 3:} $\h_{i-1}(\lk_{\Gamma_2} F;K) \ne 0$. We have
$$L + (x_i~|~ i\in N(4))=I + N(\{1,2\})\cdot N(3) + N(1)\cdot N(2)+(x_i~|~ i\in N[4]).$$
Let $H= I+N(1)\cdot N(2)+(x_i~|~ i\in N[4])$, then 
$$L + (x_i~|~ i\in N(4))=(H+(x_i~|~ i\in N(\{1,2\})))\cap(H+(x_i~|~ i\in N(3))).$$
Let $\gamma_1$ and $\gamma_2$ be the simplicial complexes which are corresponding to $H + (x_i ~| ~ i \in N (\{1,2\}) = (I+(x_i~|~ i\in N(\{1,2,4\})))$ and $H + (x_i ~| ~i \in N(3)) = (I+N(1)\cdot N(2)+(x_i~|~ i\in N(\{3,4\})))$. Then, $\Gamma_2=\gamma_1\cup\gamma_2$. Applying the Mayer-Vietoris sequence again, we also deduce three subcases:

Subcase 3.1: $\h_{i-2} (\lk_{\gamma_1} F \cap \lk_{\gamma_2} F;K) \ne 0$. One can see that $\gamma_1\cap \gamma_2$ is exactly the simplicial complex of 
\begin{multline*}
(I+(x_i~|~ i\in N(\{1,2,4\})))+(I+N(1)\cdot N(2)+(x_i~|~ i\in N(\{3,4\})))\\=I+(x_i~|~ i\in N(\{1,2,3,4\})),
\end{multline*}
and we deduce our statement as in the Case 1.

Subcase 3.2: $\h_{i-1} (\lk_{\gamma_1} F;K) \ne 0$. Similar to Case 2, we also have a contradiction.

Subcase 3.3: $\h_{i-1} (\lk_{\gamma_2} F;K) \ne 0$. We have 
\begin{multline*}
(I+N(1)\cdot N(2)+(x_i~|~ i\in N(\{3,4\})))\\=(I+(x_i~|~i\in N(\{1,3,4\}))\cap (I+(x_i~|~i\in N(\{2,3,4\})).
\end{multline*}
Let $\delta_1$ and $\delta_2$ be the simplicial complexes which are corresponding to $(I+(x_i~|~ i\in N(\{1,3,4\})))$ and $(I+(x_i~|~ i\in N(\{2,3,4\})))$. Applying the Mayer-Vietoris sequence again, we also deduce that either $\h_{i-1}(\lk_{\delta_1} F \cap \lk_{\delta_2} F;K)\ne 0$ or $\h_i(\lk_{\delta_1} F;K)\ne 0$ or $\h_i(\lk_{\delta_1} F;K)\ne 0$. Moreover, $\delta_1\cap \delta_2$ is also exactly the simplicial complex of $(I+(x_i~|~ i\in N(\{1,2,3,4\})))$. With the same argument as in the Case 1 and the Case 2, we have the desired conclusion.
\end{proof}

\begin{lem}\label{red4}  If $x^\a\notin I^{(3)}$ with $\D_\a(I^{(3)})\ne\D_\a(I^{3})$, then $\reg(I^3) \le \reg(I^{(3)}).$
\end{lem}
\begin{proof} By Lemma \ref{Key3} and Lemma \ref{red2}, we can assume that $x^\a=x_1x_2x_3\cdot f$ where $123$ is a $3$-cycle of $G$ and a monomial $f\in S$ with $\supp(f)$ is an independent set of $G$. First, we have
	 
\smallskip
\noindent\textbf{Claim B:} 
\begin{enumerate}
	\item $\supp(f)\cap \{1,2,3\}=\emptyset$.
	\item If $r\in\supp(f)$ then $x_r^2$ is not a divisor of $f$.
\end{enumerate}

Now, let $\supp(f)=\{4,\ldots,t\}$ for some $t\ge 4$. Assume that $r, s \in \supp f$ are such that $r \in N({1,2,3})$ while $s \notin N({1,2,3})$. Since $G$ is connected and has no isolated vertices, there exists a neighbor of $s$, called $u$. Since $\supp f$ is an independent set, $u \notin \supp \a$. Furthermore, $x_u \in \sqrt{I^3:x^\a}$ as $x_u x_s x_1x_2x_3x_r \in I^3.$ By Lemma \ref{red0} we are done by induction. Thus, we may assume that either  $\{4,\ldots, t\}\cap N(\{1,2,3\})=\emptyset$ or $\{4,\ldots, t\}\subseteq N(\{1,2,3\})$.

Futhermore, we also may assume that $(N(i)\cap N(j))\setminus\{1,2,3\}=\emptyset$ for all $4\le i\ne j\le t$;  $(N(\{1,2,3\})\cap N(\{4,\ldots,t\}))\setminus\{1,2,3\}=\emptyset$; and $N(1)\cap N(2)\cap N(3)=\emptyset$ by if either $u\in (N(i)\cap N(j))\setminus\{1,2,3\}$ or $u\in N(1)\cap N(2)\cap N(3)$ or $u\in (N(\{1,2,3\})\cap N(\{4,\dots,t\}))\setminus\{1,2,3\}=\emptyset$, then $x_u\in\sqrt{I^3:x^\a}$ and $u\notin\supp(\a)$ and we are done by induction and Lemma \ref{red0}.

\smallskip

\noindent \textbf {Case 1:} $\{4,\ldots, t\}\cap N(\{1,2,3\})=\emptyset$. Denote $N^*(i) = N(i) \setminus \{1,2,3\}$. In this case, we have

\smallskip

\noindent\textbf{Claim C:}
	$$\sqrt{I^3 : x^\a} = I + \sum_{4\le i < j\le t} N(i) \cdot N(j)+N(\{1,2,3\}) \cdot N(\{4,\ldots,t\})+N^*(1)\cdot N^*(2)\cdot N^*(3).$$
\smallskip

There are two subcases: 

Subcase 1.1. If $t \ge 5$. Let $$H= I + \sum_{4\le i < j\le t, (i,j) \neq (t-1,t)} N(i) \cdot N(j)+N^*(1)\cdot N^*(2)\cdot N^*(3)+N(\{1,2,3\}) \cdot N(\{4,\ldots,t\}).$$ Then, 
$$\sqrt{I^3:x^\a} = (H + (x_i~|~i\in N(t-1))) \cap (H + (x_i~|~i\in N(t))).$$ Let $\Gamma_1$ and $\Gamma_2$ be the simplicial complexes which are corresponding to $H + (x_i~|~i\in N(t-1))$ and $H + (x_i~|~ i\in N(t))$. Hence, $$\Delta_\a(I^3) = \Gamma_1 \cup \Gamma_2.$$
Moreover, $\Gamma_1\cap\Gamma_2$ is the simplicial complex of $H + (x_i~|~i\in N(\{t-1,t\}))$. Using the above Claim A, one can see that $\Gamma_1$ is a cone over $t-1$, $\Gamma_2$ and $\Gamma_1\cap\Gamma_2$ are cones over $t$. By $F\cap\supp(\a)=\emptyset$, $t, t-1\notin F$. Applying the Mayer-Vietoris sequence, this implies that either $\h_{i-1}(\lk_{\Gamma_1} F) \neq 0$ or  $\h_{i-1}(\lk_{\Gamma_2} F) \neq 0$ or $\h_{i-2}(\lk_{\Gamma_1} F\cap \lk_{\Gamma_2} F) \neq 0$, which is a contradiction.

Subcase 1.2. If $t = 4$. Let $H = I + N^*(1)\cdot N^*(2)\cdot N^*(3)$. Then, 
$$\sqrt{I^3:x^\a} =(H+(x_i~|~ i\in N(\{1,2,3\})))\cap (H+(x_i~|~ i\in N(4))).$$ 
Let $\Gamma_1$ and $\Gamma_2$ be the simplicial complexes which are corresponding to $H + (x_i~|~i\in N(\{1,2,3\}))$ and $H + (x_i~|~ i\in N(4))$. Hence, 
$$\Delta_\a(I^3) = \Gamma_1 \cup \Gamma_2.$$
Moreover, $\Gamma_1\cap\Gamma_2$ is also the simplicial complex of $H + (x_i~|~i\in N(\{1,2,3,4\}))$. Using Claim A, we can see that $\Gamma_2$ and $\Gamma_1\cap\Gamma_2$ are cones over $4$. By $F\cap\supp(\a)=\emptyset$, $4\notin F$. Applying the Mayer-Vietoris sequence, this implies that $\h_{i-1}(\lk_{\Gamma_1} F) \neq 0$. On the other hand,
\begin{align*}
H+(x_i~|~i\in N(\{1,2,3\})=I + (x_i~|~i\in N(\{1,2,3\}).
\end{align*}
Let $\b = x_1^2x_2^2x_3$. Then $\sqrt{I^3:x^\b} = I + (x_i ~|~ i \in N(\{1,2,3\}))$. Thus, 
$$|\b|+i\le \reg(S/I^3)=|\a|+i,$$
which is a contradiction, as $|\b| > |\a|$.

\smallskip

\noindent \textbf{Case 2:} $\{4,\ldots, t\}\subseteq N(\{1,2,3\})$. If $t \ge 5$, we have a contradiction by

\smallskip

\noindent\textbf{Claim D:} $\Delta_\a(I^3)$ is a cone over $4$.
\smallskip

Thus, $t = 4$. There are two subcases:

Subcase 2.1: $|N(4) \cap \{1,2\}| = 1$. Assume that $x_1x_4 \in I$ and $x_2x_4, x_3x_4\notin I$. One can see that 
$$\sqrt{I^3:x^\a} = I + N(2)\cdot N(3) + N(\{1,2,3\})  \cdot N(4)+(x_1).$$
Let $J = I+(x_1) + N(2)\cdot N(3)$. Then 
$$\sqrt{I^3:x^\a} = (J  + (x_i~|~i\in N(4)))  \cap (J + (x_i~|~i\in N(\{1,2,3\}))).$$
Let $\Gamma_1,\Gamma_2$ be the corresponding simplicial complexes of $J  + (x_i~|~i\in N(4))$  and $J + (x_i~|~i\in N(\{1,2,3\}))$. Then 
$$\Delta_\a(I^3) = \Gamma_1 \cup \Gamma_2.$$
Note that $\Gamma_1 \cap \Gamma_2$ is the simplicial complex of $J + (x_i~|~i\in N(\{1,2,3,4\}))$. 
It is noted that $4\notin F$ by $4\in\supp(\a)$. Since $\Gamma_1$ is a cone over $4$, applying the Mayer-Vietoris sequence, this implies that either $\h_{i-1}(\lk_{\Gamma_2} F;K) \neq 0$ or $\h_{i-2} (\lk_{\Gamma_1 \cap \Gamma_2} F;K) \neq 0$.

Subcase 2.1.1. $\h_{i-1} (\lk_{\Gamma_2} F;K) \neq 0$. Let $\b = x_1^2x_2^2x_3$, then 
$$\sqrt{I^3:x^\b} = I + (x_i~|~i\in N(\{1,2,3\})) = J +  (x_i~|~i\in N(\{1,2,3\})).$$
Thus, $|\b|+i\le \reg(S/I^3)=|\a|+i,$ which is a contradiction, as $|\b| > |\a|$.

Subcase 2.1.2. $\h_{i-2}(\lk_{\Gamma_1 \cap \Gamma_2} F;K) \neq 0.$ Let $\b = x_1^2x_2x_3x_4$, then
$$\sqrt{I^3:x^\b} = I + (x_i~|~i\in N(\{1,2,3,4\})) = J +(x_i~|~i\in  N(\{1,2,3,4\})).$$
Thus, $|\b|+i-1\le \reg(S/I^3)=|\a|+i=|\b|+i-1$. Then $(\b,i-1)$ is also an extremal exponent with respect to $I^3$. Furthermore, $x_1^2x_2x_3x_4 = (x_1x_2x_3) \cdot (x_1x_4)$ is the product of a triangle and an edge, which was treated in Lemma \ref{red3}. From this, we deduce that $\reg I^3 \le \reg I^{(3)}$ as required.

Subcase 2.2: $|N(4) \cap \{1,2\}| = 2$. Assume that $x_1x_4, x_4x_2 \in I$ and $x_3x_4 \notin I$. In this case, we have 
$$\sqrt{I^3:x^\a} = I + N(4) \cdot N(\{1,2,3\}) + N(3) \cdot N(\{1,2\}) + (x_1,x_2).$$ 
Let $H = I + (x_1,x_2) + N(3) \cdot N(\{1,2\})$. Then $\sqrt{I^3:x^\a} = (H + (x_i~|~i\in N(4))) \cap (H + (x_i~|~i\in N \{1,2,3\}))$. Let $\delta_1,\delta_2$ be the corresponding simplicial complexes of $H  + (x_i~|~i\in N(4))$  and $H+ (x_i~|~i\in N(\{1,2,3\}))$. Then 
$$\Delta_\a(I^3) = \delta_1 \cup \delta_2.$$
Note that $\delta_1 \cap \delta_2$ is the simplicial complex of $H + (x_i~|~i\in N(\{1,2,3,4\}))$. Applying again the Mayer-Vietoris sequence, we deduce that either $\h_{i-1}(\lk_{\delta_1}F;K) \neq 0$ or $\h_{i-1}(\lk_{\delta_2}F;K) \neq 0$ or $\h_{i-2}(\lk_{\delta_1 \cap \delta_2} F;K) \neq 0$, where $F\in \Delta_\a(I^3)$ such that $F\cap\supp(\a)=\emptyset$ and $\h_{i-1}(\lk_{\Delta_\a(I^3)} F;K)\ne 0$. 

Subcase 2.2.1: $\h_{i-1} (\lk_{\delta_2} F;K) \neq 0$. It is noted that $$H+ (x_i~|~i\in N(\{1,2,3\}))=I+(x_i~|~i\in N(\{1,2,3\})),$$ which gives rise to a contradiction with similar argument as in the subcase 2.1.1. 

Subcase 2.2.2: $\h_{i-2}(\lk_{\delta_1 \cap \delta_2} F;K) \neq 0$. We also have $$H + (x_i~|~i\in N(\{1,2,3,4\}))=I + (x_i~|~i\in N(\{1,2,3,4\})),$$ which is the subcase 2.1.2.

Subcase 2.2.3: $\h_{i-1}(\lk_{\delta_1} F;K) \neq 0$. In this case, we have
 \begin{align*}
 H+ (x_i~|~i\in N(4))&=I+(x_i~|~i\in N(4))+N(3) \cdot N(\{1,2\})\\
 &=(I+(x_i~|~i\in N(\{3,4\})))\cap (I+(x_i~|~i\in N(\{1,2,4\}))).
 \end{align*}
Let $\gamma_1,\gamma_2$ be the corresponding simplicial complexes of $I+(x_i~|~i\in N(\{3,4\}))$  and $I+(x_i~|~i\in N(\{1,2,4\}))$. Then, $\delta_1= \gamma_1\cup \gamma_2$ and $\gamma_1 \cap \gamma_2$ is the simplicial complex of $I + (x_i~|~i\in N(\{1,2,3,4\}))$. Similar to the subcases 2.2.1 and 2.2.2 and $\gamma_1$ is a cone over $4$, we obtain our assertion. Our proof will be completed once we prove Claim B, Claim C, and Claim D, which is done in the sequence. 
\end{proof}

\begin{proof}[Proof of Claim B]
(1) Assume that $1\in \supp(f)\cap \{1,2,3\}$. By $x^\a\notin IJ_1(G)$, $1r\notin G$ for any $r\in\supp(f)\setminus\{1,2,3\}$ i.e. $x_1\notin \sqrt{I^3:x^\a}$. Moreover, $x_r\in \sqrt{I^3:x^\a}$ for all $r\in N(1)$ (by $(x_rx_1)^2x_2x_3\in I^3$). And, the variable $x_1$ will be not appeared in any minimal squarefree monomial generator of $\sqrt{I^{3} : x^\a}$. In fact, by Lemma \ref{radical_colon}, if $g$ is a minimal squarefree monomial generator of $\sqrt{I^{3} : x^\a}$ then $g|\dfrac{x_{e_1}x_{e_2}x_{e_3}}{\gcd(x_{e_1}x_{e_2}x_{e_3}, x^\a)}$ for some edges $e_1, e_2, e_3$ of $G$. By $x_1|g, x_1^2|x^\a$ then $x_1^3|x_{e_1}x_{e_2}x_{e_3}$. But $x_r\in \sqrt{I^3:x^\a}$ for all $r\in N(1)$, $g$ must be $x_1$, which is a contradiction. Using Claim A, $\Delta_\a(I^3)$ is a cone over $1$, a contradiction. 

(2) Assume that $r\in\supp(f)$ and $x_r^2 | f$, then $x_i\in\sqrt{I^3:x^\a}$ for all $i\in N(r)$ by $(x_ix_r)^2x_1x_2x_3\in I^3$. If $x_r\in \sqrt{I^3:x^\a}$, then $1r,2r,3r \in G$, as there is no edges in $\supp(f)$. Thus, $x^\a\in J_2(G)\subseteq I^{(3)}$, a contradiction. If $x_r\notin \sqrt{I^3:x^\a}$, similarly, the variable $x_r$ will be not appeared in any minimal squarefree monomial generator of $\sqrt{I^{3} : x^\a}$. Then $\Delta_\a(I^3)$ is a cone over $r$ (by again the Claim A), which is also a contradiction. This completes our Claim B.
\end{proof}	

\begin{proof}[Proof of Claim C:] 
It is easy to see that the left hand side contains the right hand side. To prove the reverse inclusion, let $g$ be a minimal monomial generator of $\sqrt{I^3: x^\a}$. By Lemma \ref{radical_colon}, there exist three edges $e_1,e_2,e_3$ of $G$ such that 
$$g = \sqrt{\dfrac{x_{e_1}x_{e_2}x_{e_3}}{\gcd(x_{e_1}x_{e_2}x_{e_3},x^\a)}}.$$ 
If $\supp e_i \cap \supp \a = \emptyset$, then $g \in I$, thus we may assume that $\supp e_i \cap \supp \a \neq \emptyset$ for all $i = 1,2,3$. There are two cases:

\smallskip
\noindent\textbf{Case 1:} $\supp e_i \subseteq \supp \a$ for some $i$. Assume that $\supp e_1 \subseteq x^\a$. This implies that $\supp e_1 \subseteq \{1,2,3\}$. We may assume that $e_1 = 12$. Thus, $g = \sqrt{\dfrac{x_{e_2} x_{e_3}}{\gcd (x_{e_2}x_{e_3}, x_3x_4\cdots x_t)}}.$ Let $i = \supp e_2 \cap \{3,4,\ldots,t\}$ and $j = \supp e_3 \cap \{3,4,\ldots,t\}$. Then $i \neq j$ and $g = x_{e_2}x_{e_3} / x_i x_j \in N(i) \cdot N(j).$ 

\smallskip
\noindent\textbf{Case 2:} $|\supp e_i \cap \supp \a| = 1$ for all $i = 1,2,3$. Let $i = \supp e_1 \cap \supp \a$, $j = \supp e_2 \cap \supp a$, and $k = \supp e_3 \cap \supp a$. Then $i, j, k$ must be distinct, and $g = \sqrt{x_{e_1} x_{e_2} x_{e_3} / x_i x_j x_k}$. If either $i,j, k \ge 4$, say $k \ge 4$, then $g \in N(\{1,2,3\}) \cdot N(k)$. If $i,j,k \le 3$, then $\{1,j,k\} = \{1,2,3\}$, and $g \in N^*(1) \cdot N^*(2) \cdot N^*(3)$. This completes Claim C.
\end{proof}
\begin{proof}[Proof of Claim D]
Assume that $g$ is a minimal monomial generator of $\sqrt{I^3:x^\a}$ such that $x_4 | g$. By Lemma \ref{radical_colon}, there exist three edges $e_1, e_2, e_3$ of $G$ such that 
$$g = \sqrt{\dfrac{x_{e_1}x_{e_2}x_{e_3}}{\gcd(x_{e_1}x_{e_2}x_{e_3},x^\a)}}.$$
In particular, $x_4^2 | x_{e_1}x_{e_2}x_{e_3}$. We may assume that $e_1, e_2$ contains $4$. Let $x_{e_1} = x_4 x_s$, $x_{e_2} = x_4x_u$, then 
$g = x_4 \sqrt{\dfrac{x_s x_ux_{e_3}}{\gcd (x_s x_ux_{e_3},x_1x_2x_3 x_5 \cdots x_t)}}$. Since $x_s \in \sqrt{I^3:x^\a}$, and $g$ is minimal, $x_s | x_1x_2x_3 x_5 \cdots x_t$. From $s \in N(4)$, $s \in \{1,2,3\}$. We may assume that $s = 1$. Similarly, $u \in \{2,3\}$.  Assume that $u = 2$, then $g=x_4 \cdot \sqrt{x_{e_3} / \gcd (x_{e_3},x_3 x_5 \cdots x_t)}$. Since $g$ is minimal, there is $j\in\supp(e_3)\cap \{3,5,\cdots, t\}$. Therefore, $g=x_4 \cdot (x_{e_3}  /x_j)$. Furthermore, $(x_{e_3} / x_j) x^\a\in I^3$ i.e. $x_{e_3} / x_j\in \sqrt{I^3:x^\a}$, which is a contradiction to the minimality of $g$. 
\end{proof}

We are now ready for the main result of the section.

\begin{thm}\label{Main2} Let $G$ be a simple graph and $I$ its edge ideal. Then 
$$\reg I^{(3)} = \reg I^3.$$
\end{thm}
\begin{proof} By Theorem \ref{thirdpower_in}, it suffices to prove that $\reg I^3 \le \reg I^{(3)}$. Let $(\a,i)$ be an extremal exponent of $I^3$. Then $x^\a \notin I^3$. We have the following cases.

\noindent\textbf{Case 1:} $x^\a \notin I^{(3)}$ and $\Delta_\a(I^3) = \Delta_\a(I^{(3)}$. The conclusion follows from Lemma \ref{extremal_red}.

\noindent\textbf{Case 2:} $x^\a \notin I^{(3)}$ and $\Delta_\a(I^3) \neq \Delta_\a(I^{(3)}$. The conclusion follows from Lemma \ref{red4}.

\noindent\textbf{Case 3:} $x^\a \in I^{(3)}$. The conclusion follows from Theorem \ref{expansion_third}, and Lemmas \ref{red1}, \ref{red2}, and \ref{red3}.
\end{proof}

\section{Applications}
In this section, we give some applications of our main results. First, we establish Conjecture B for $s = 2$ and $s = 3$, which extends work of Banerjee and Nevo \cite{BN}.

\begin{thm}\label{Main3} Let $G$ be a simple graph and $I$ its edge ideal. Then 
$$\reg I^s \le \reg I +  2s -2$$
for $s = 2, 3.$
\end{thm}
\begin{proof}
By \cite[Theorem 3.6]{F4}, we have
$$\reg I^{(s)} \le \max \{ \reg (I^{(s)} + I^{s-1}), \reg I + 2s - 2\}.$$
For $s = 2, 3$, note that $I^{(s)} \subseteq I^{s-1}$. Thus, for $s = 2$, by Theorem \ref{Main1}, 
$$\reg I^2 = \reg I^{(2)} \le \max \{ \reg I, \reg I + 2\} = \reg I+ 2.$$
For $s = 3$, by Theorem \ref{Main2}, 
$$\reg I^{3} = \reg I^{(3)} \le \max \{ \reg I^2, \reg I + 4\} = \reg I + 4.$$
This completes our proof.
\end{proof}

\begin{rem} Theorem \ref{Main3} shows that bounds for edge ideals generalize to bounds for second and third powers of edge ideals. For example, combinatorial bound given by Woodroofe \cite{W} carries over to bounds for regularity of the second/third powers of an edge ideal.
\end{rem}

For symbolic powers, we prove

\begin{thm}\label{Main4} Let $G$ be a simple graph and $I$ its edge ideal. Then 
$$\reg I^{(s)} \le \reg I + 2s - 2$$
for $s = 2,3,4.$
\end{thm} 
The case $s \le 3$ was already proved in Theorem \ref{Main3}. Let $s = 4$. Using \cite[Theorem 3.6]{F4}, we need to bound $\reg (I^{(4)} + I^3).$  First, we have

\begin{lem}\label{fourth_third} Denote $J_3(G)$ the ideal of $S$ generated by $x_C$ where $C$ is a clique of size $5$ in $G$. Then $$I^{(4)} + I^3 = I^3 +  J_1(G) J_1(G) + J_3(G).$$ 
\end{lem}
\begin{proof} Using Lemma \ref{differential_criterion}, it is easy to see that the left hand side contains the right hand side. Conversely, let $f$ be a minimal generator of $I^{(4)}$. By Lemma \ref{restriction}, we may assume that $\supp f = V(G) = [n]$, and $G$ has no isolated vertices. If the matching number of $G$ is at least $3$, then $f \in I^3$. Thus, we may assume that the matching number of $G$ is at most $2$. There are two cases:

\smallskip

\noindent\textbf{Case 1:} $G$ contains an induced $5$-cycle, called $12345$. Write $f = x_1x_2x_3x_4x_5 g$. By Lemma \ref{differential_criterion}, this implies that $x_1 x_3g \in I$. Since $12345$ is an induced $5$-cycle, $x_1x_3 \notin I$, thus either $x_1g$ or $x_3g\in I$. We may assume that $x_1g \in I$. This implies that $f = (x_1g) (x_2x_3)(x_4x_5) \in I^3$. 

\smallskip

\noindent\textbf{Case 2:} $G$ does not contains any $5$-cycle. Since the matching number of $G$ is at most $2$, by \cite[Theorem 5.5.3]{D}, $G$ is a perfect graph. The conclusion follows from \cite[Theorem 3.10]{Su}.
\end{proof}

We are now ready for

\begin{proof}[Proof of Theorem \ref{Main4}] By Theorem \ref{Main3}, we may assume that $s = 4$. By \cite[Theorem 3.6]{F4}, we have 
$$\reg I^{(4)} \le \max \{ \reg (I^{(4)} + I^3), \reg I + 6\}.$$
Let $H = I^{(4)} + I^3$. By Lemma \ref{fourth_third}, $H = I^3 + J_1(G) J_1(G) + J_3(G)$. Fix $\a \in \NN^n$ such that $x^\a \notin H$. We will prove that $\Delta_\a(H) = \Delta_\a(I^3)$. Assume by contradiction that $\Delta_\a(H)\neq \Delta_\a(I^3)$. By Lemma \ref{Key1}, there exists a minimal squarefree monomial generator $f$ of $\sqrt{H^3:x^\a}$ such that $f \notin \sqrt{I^3:x^\a}$. By Lemma \ref{radical_colon}, we have two cases:

\smallskip
\noindent\textbf{Case 1:} $f = \sqrt{x_1x_2x_3x_4x_5/\gcd(x_1x_2x_3x_4x_5,x^\a)}$, where $12345$ forms a clique of size $5$ in $G$. If $\deg f \ge 2$, then $x_ix_j| f$ and $f \in I \subseteq \sqrt{I^3:x^\a}$, which is a contradiction. If $\deg f = 1$, say $f = x_1$, then $x_2x_3x_4x_5 | x^\a$. But then $x_1 \in \sqrt{I^3:x^\a}$, which is also a contradiction.

\smallskip
\noindent\textbf{Case 2:} $f = \sqrt{x_1x_2x_3x_4x_5x_6/\gcd(x_1\cdots x_6,x^\a)}$ where $123$ and $456$ are triangles in $G$. If $\{1,2,3\} \cap N(\{4,5,6\}) \neq \emptyset$, then $x_1\cdots x_6 \in I^3$, which is a contradiction. Thus, we may assume that $\{1,2,3\} \cap N(\{4,5,6\}) = \emptyset.$ Since $f \notin \sqrt{I^3:x^\a}$, $f\notin I$. Since $f | x_1\cdots x_6$, $\deg f \le 2$. We may assume that $\supp f \subseteq \{1,4\}$. This implies that $x_2x_3x_5x_6 |x^\a$. But, $x_1^2 (x_2x_3x_5x_6) = (x_1x_2)(x_1x_3)(x_5x_6) \in I^3$ and $x_4^2(x_2x_3x_5x_6) = (x_2x_3)(x_4x_5)(x_4x_6) \in I^3$. In other words, $x_1, x_4 \in \sqrt{I^3:x^\a}$. This implies that $f \in \sqrt{I^3:x^\a}$, which is a contradiction.

Thus, we have $\Delta_\a(H) = \Delta_\a(I^3)$ for all $\a \in \NN^n$ such that $x^\a \notin H$. By Lemma \ref{extremal_red}, $\reg H \le \reg I^3 \le \reg I + 4$, where the last inequality follows from Theorem \ref{Main3}. This concludes our proof.
\end{proof}

\begin{rem} Theorem \ref{Main4} implies that bounds for regularity of edge ideals generalize to bounds for regularity of $s$-th symbolic powers of edge ideals ($s \le 4$). In particular, upper bounds given by Herzog and Hibi \cite{HH}, and Fakhari and Yassemi \cite{FY} hold for $s$-th symbolic powers of edge ideals ($s \le 4$).
\end{rem}

By Theorem \ref{main_thm} and Theorem \ref{Main4}, we obtain a formula of the regularity of small symbolic powers of edge ideals of some new classes of graphs.

\begin{cor}Let $G$ be a simple graph and $I=I(G)$. 

\begin{enumerate}
    \item If $I^s$ has a linear resolution (for $s=2$ or $3$) 
    then $\reg(I^{(s)})= 2s$.
    \item If $\reg I = \mu(G) + 1$, where $\mu(G)$ is the induced matching number of $G$, then $\reg I^{(s)} = 2s + \mu(G) -1$ for $s = 2,3,4.$
\end{enumerate}
\end{cor}
\begin{proof} The first statement is a direct consequence of Theorem \ref{main_thm}. The second statement follows from Theorem \ref{Main4} and Lemma \ref{restriction_inq}.
\end{proof}

\begin{rem} \begin{enumerate}
    \item There is an infinite family of graphs $G$ with the property that although each edge ideal $I(G)$ does not have a linear resolution, a higher power does (see \cite{N}).
    \item There are classes of graphs for which $\reg I = \mu(G) + 1$, while regularity of their symbolic powers was not known. Such examples include the class of very-well covered graphs \cite{JS} and weakly chordal graphs \cite{NV1}.
\end{enumerate} 
\end{rem}

We end the paper with the following remark.
\begin{rem} From the proof of Theorem \ref{Main4}, we see that 
$$\reg (I^3 + J_1(G) J_1(G))  \le \reg  I^3.$$
Note that $I^3 + J_1(G)J_1(G) = \overline{I^3}$ is the integral closure of $I^3$. The regularity of integral closure of powers of edge ideals will be studied in detail in subsequent work.
\end{rem}
\subsection*{Acknowledgment} This paper was done while the first author was visiting Vietnam Institute for Advanced Study in Mathematics (VIASM).  He would like to thank the VIASM for hospitality and financial support. The first author also is partially supported by the project B2022-SPH-02 of Vietnam Ministry of Education and Training. The fourth author  was funded by Vingroup Joint Stock Company and supported by the Domestic Master/ PhD Scholarship Programme of Vingroup Innovation Foundation (VINIF), Vingroup Big Data Institute (VINBIGDATA).


\begin{thebibliography}{2}
\bibitem[BBH]{BBH}
A. Banerjee, S. Beyarslan, H. T. Ha, 
{\em Regularity of edge ideals and their powers}, 
Springer Proceedings in Mathematics \& Statistics {\bf 277} (2019), 17--52.

\bibitem[BH]{BH}
W. Bruns and J. Herzog,
\emph{Cohen-Macaulay rings. Rev. ed.}.
Cambridge Studies in Advanced Mathematics {\bf 39}, Cambridge University Press (1998).


\bibitem[BN]{BN} 
A. Banerjee and E. Nevo,
{\em Regularity of edge ideals via suspension},
arXiv:1908.03115, August 2019.

\bibitem [Cu]{Cu} 
S. Cutkosky, 
{\em Irrational asymptotic behaviour of Castelnuovo-Mumford regularity},
J. Reine Angew. Math. {\bf 522} (2000), 93--103.

\bibitem[CHHKTT]{CHHKTT}
G. Caviglia, H. T. Ha, J. Herzog, M. Kummini, N. Terai, and N. V. Trung,
\emph{Depth and regularity modulo a principal ideal},
J. Algebraic Combinatorics. {\bf 40} (2018), 353--374.

\bibitem[CHT]{CHT}
S.D. Cutkosky, J. Herzog and N.V. Trung,
{\em Asymptotic behaviour of the Castelnuovo-Mumford regularity},
 Compositio Math. {\bf 118} (1999), no. 3, 243--261. 
 
 \bibitem[D]{D}
 R. Diestel, \emph{Graph theory, 2nd. edition,} Springer: Berlin/Heidelberg/New York/Tokyo, 2000.
 
 \bibitem[DHNT]{DHNT}
 L. X. Dung, T. T. Hien, H. D. Nguyen, and T. N. Trung,
 {\em Regularity and Koszul property of symbolic powers of monomial ideals},
 arXiv:1903.09026, March 2019.
 
\bibitem[E]{E}
D. Eisenbud,
\emph{Commutative algebra. With a view toward algebraic geometry},
Graduate Texts in Mathematics {\bf 150}. Springer-Verlag, New York (1995).
 
\bibitem[F1]{F1}
S. A. Seyed Fakhari, 
{\em Regularity of symbolic powers of edge ideals of unicyclic graphs},
 J. Algebra {\bf 541} (2020), 345--358.

\bibitem[F2]{F2}
S. A. Seyed Fakhari, 
{\em Regularity of symbolic powers of edge ideals of chordal graphs},
preprint.

\bibitem[F3]{F3}
S. A. Seyed Fakhari, 
{\em Regularity of symbolic powers of edge ideals of Cameron-Walker graphs},
arXiv:1907.02743, July 2019.

\bibitem[F4]{F4}
S. A. Seyed Fakhari, 
{\em On the regularity of small symbolic powers of edge ideals of graphs},
arXiv:1908.10845, August 2019.

\bibitem[FY]{FY}
S. A. Seyed Fakhari, and S. Yassemi,
{\em Improved bounds for the regularity of edge ideals of graphs}, Collectanea Math. {\bf 69} (2018), 249--262.

\bibitem[GHOS]{GHOS}
Y. Gu, H. T. Ha, J. L. O'Rourke, and J. W. Skelton,
{\em Symbolic powers of edge ideals of graphs},
arXiv:1805.03428, May 2018.


\bibitem[HH]{HH}
J. Herzog, T. Hibi,
{\em An upper bound for the regularity of powers of edge ideals},
Math. Scandinavica, {\bf 126} (2020), 165--169.

\bibitem [HHT]{HHT} 
J. Herzog, L. T. Hoa, N. V. Trung, 
{\em Asymptotic linear bounds for the Castelnuovo-Mumford regularity}, 
Trans. Amer. Math. Soc. {\bf 354} (2002), no. 5, 1793--1809.

\bibitem [HNTT]{HNTT} 
H. T. Ha, H. D. Nguyen, N. V. Trung, and T. N. Trung, 
{\em Symbolic powers of sums of ideals},
Math. Z. {\bf 294} (2020), 1499--1520. 

\bibitem [HTr]{HTr}
 L. T. Hoa, T. N. Trung, 
 {\em Castelnuovo-Mumford regularity of symbolic powers of two-dimensional square-free monomial ideals}, 
 J. Commut. Algebra {\bf 8} Number 1 (2016), 77--88.
 
 \bibitem [JK]{JK}
 A. V. Jayanthan, R. Kumar,
 {\em Regularity of Symbolic Powers of Edge Ideals}, 
 J. Pure Appl. Algebra {\bf 224} (2020), 106306. 

\bibitem[JS]{JS}
A. V. Jayanthan, S. Selvaraja,
\emph{Linear polynomial for the regularity of powers of edge ideals of very well-covered graphs}. 
arXiv:1708.06883, August 2017, J. Commut. Algebra, to appear.


\bibitem[K] {K} 
V. Kodiyalam, 
{\em Asymptotic behaviour of Castelnuovo-Mumford regularity},
 Proc. Amer. Math. Soc. {\bf 128} (2000), 407--411.

\bibitem[MTr] {MTr} 
N. C. Minh, T. N. Trung, {\em Regularity of symbolic powers and arboricity of matroids}, 
Forum Mathematicum {\bf 31} (2019), Issue 2, 465--477.  

\bibitem[MTru] {MTru} 
N. C. Minh, H. L. Truong, in preparation.  

\bibitem[N] {N} 
E. Nevo,
{\em Regularity of edge ideals of C4-free graphs via the topology of the lcm-lattice},
J. Combin. Theory Ser. A {\bf 118} (2011), no. 2, 491--501.


\bibitem[NV1] {NV1} 
H. D. Nguyen, T. Vu,
{\em Linearity defect of edge ideals and Fr$\ddot{\text{o}}$berg's theorem},
J. Algebr. Combinatorics {\bf 44} (2016), 165--199.


\bibitem[NV2] {NV2} 
H. D. Nguyen, T. Vu,
{\em Powers of sums and their homological invariants},
J. Pure Appl. Algebra {\bf 223} (2019), no. 7, 3081--3111.

\bibitem[S]{S} 
R. Stanley, 
{\em Combinatorics and Commutative Algebra}, 
2. Edition, Birkh$\ddot{\text{a}}$user, 1996.

\bibitem[SVV]{SVV}
A. Simis, W. Vasconcelos and R.H. Villarreal, 
{\em On the ideal theory of graphs},
J. Algebra {\bf 167} (1994), no. 2, 389--416.

\bibitem[Su]{Su}
S. Sullivant, 
{\em Combinatorial symbolic powers},
J. Algebra {\bf 319} (2008), no. 1, 115--142.

\bibitem[T] {T}
Y. Takayama, 
{\em Combinatorial characterizations of generalized Cohen-Macaulay monomial ideals}, 
Bull. Math. Soc. Sci. Math. Roumanie (N.S.) {\bf 48} (2005), 327--344.

\bibitem[W]{W}
R. Woodroofe,
{\em Matchings, coverings, and Castelnouvo-Mumford regularity},
J. Commut. Algebra {\bf 6} (2010), 287--304.

\end{thebibliography}
\end{document}